\documentclass[hidelinks,onefignum,onetabnum]{siamart220329}

\usepackage{lipsum}
\usepackage{amsfonts}
\usepackage{graphicx}
\usepackage{epstopdf}
\usepackage{algorithmic}
\ifpdf
  \DeclareGraphicsExtensions{.eps,.pdf,.png,.jpg}
\else
  \DeclareGraphicsExtensions{.eps}
\fi

\newsiamremark{remark}{Remark}
\newsiamremark{hypothesis}{Hypothesis}
\crefname{hypothesis}{Hypothesis}{Hypotheses}
\newsiamthm{claim}{Claim}

\headers{Constrained HK barycenters}{M.~Buze}

\title{Constrained Hellinger--Kantorovich barycenters: least-cost soft and conic multi-marginal formulations\thanks{
\funding{This work was supported by research grant EP/V00204X/1.}}}

\author{Maciej Buze\thanks{School of Mathematical Sciences, Lancaster University, Lancaster, LA1 4YW, UK
(\email{m.buze@lancaster.ac.uk}, \url{https://mbuze.github.io}).}
}
\usepackage{amsopn}

\ifpdf
\hypersetup{
  pdftitle={Constrained Hellinger-Kantorovich barycenters: least-cost soft and conic multi-marginal formulations},
  pdfauthor={M. Buze}
}
\fi

\usepackage{enumitem}
\usepackage{makecell}
\usepackage{graphicx}
\usepackage[labelformat=simple]{subcaption}

\usepackage[english]{babel}
\usepackage[T1]{fontenc}

\newsiamremark{example}{Example}

\newcommand{\cb}{}

\def\F{\mathcal{F}}

\def\H{\mathcal{H}}
\def\E{\mathcal{E}}
\def\Rc{\mathcal{R}}
\def\R{\mathbb{R}}

\def\d{\mathsf{d}}
\def\h{\mathsf{h}}

\def\M{\mathcal{M}}

\begin{document}

\maketitle

% REQUIRED
\begin{abstract}
  We show that the problem of finding the barycenter in the Hellinger–-Kantorovich setting admits a least-cost soft multi-marginal formulation, provided that a one-sided hard marginal constraint is introduced. The constrained approach is then shown to admit a conic multi-marginal reformulation based on defining a single joint \emph{multi-marginal perspective cost function} in the conic multi-marginal formulation, as opposed to separate two-marginal perspective cost functions for each two-marginal problem in the coupled-two-marginal formulation, as was studied previously in literature. We further establish that, as in the Wasserstein metric, the recently introduced framework of unbalanced multi-marginal optimal transport can be reformulated using the notion of the least cost. Subsequently, we discuss an example when input measures are Dirac masses and numerically solve an example for Gaussian measures. Finally, we also explore why the constrained approach can be seen as a natural extension of a Wasserstein space barycenter to the unbalanced setting.
\end{abstract}

% REQUIRED

% REQUIRED

\section{Introduction}
The theory of optimal transportation, dating back to Gaspard Monge's work in 1781 \cite{monge1781memoire}, continues to develop at pace as one of the fundamental mathematical theories with an ever-growing list of diverse applications in fields such as economics, computer vision, image processing and machine learning -- see monographs in \cite{villani2009optimal,S15,PC19} for a broad overview.

A central challenge in many applications concerns finding a representative, or \emph{barycentric} (probability) distribution, which provides some average description of a given set of distributions. The basic optimal transport approach to this problem is to find the barycenter by minimizing the sum of weighted two-marginal optimal transport costs between the barycenter and each input distributions. It was subsequently shown that an equivalent and computationally favourable approach is to instead solve a single \emph{least-cost multi-marginal} optimal transport problem \cite{CE10,AC11}. 

If the input distributions do not all have equal mass, an unbalanced barycenter can be found via a recourse to the emerging theory of unbalanced optimal transportation \cite{CPSV16,KMV16,LMS18,liero2023fine}. This however, can be done in a number of ways, depending on how one penalises mass deviations, what cost function is employed and whether one wishes to consider the conic formulation. Recent important contributions in this area concern:
\setlist[enumerate]{wide=0pt, widest=99, labelwidth=60pt, leftmargin=70pt, align=left}
\begin{enumerate}
    \item[\textbf{(HKB):}] \noindent the study of coupled-two-marginal Hellinger-Kantorovich barycenter and its equivalence to a conic least-cost multi-marginal formulation, as well as the non-existence of an equivalent soft multi-marginal formulation \cite{FMS21,CP21,BMS23}; \label{framework1}
    \item[\textbf{(}${\rm \textbf{UMOT}}_{\varepsilon}$\textbf{):}] \noindent the study of entropy-regularised unbalanced multi-marginal optimal transport and its equivalence to the entropy regularised coupled-two-marginal formulation where the second marginal is fixed  \cite{BvLNS23} -- notably, this is done in an extended space setting, as opposed to the least-cost setting known from the balanced case. \label{framework2}
\end{enumerate}

In particular, apart from the explicitly different starting coupled-two-marginal formulation for \hyperref[framework1]{\textbf{(HKB)}} and \hyperref[framework2]{\textbf{(}${\rm \textbf{UMOT}}_{\varepsilon}$\textbf{)}} (putting aside the question of entropic regularisation), the relation between the two frameworks, especially with regards to their multi-marginal reformulations, has not been thoroughly explored. At the same time, multi-marginal approaches to (unbalanced) optimal transport continue to attract considerable attention both from a theoretical point of view and for applications, c.f., among many others, \cite{pass2015multi, haasler2021multimarginal, ba2022accelerating, friesecke2023gencol,altschuler2023polynomial, trillos2023multimarginal,pass2024general}, so it is vital to study the mathematical underpinning of its unbalanced variant.

\subsection{Outline of the paper and its contributions}

In this paper we introduce and discuss the notion of a \emph{constrained Hellinger-Kantorovich} barycenter, which, in its basic form and similarly to \hyperref[framework2]{\textbf{(}${\rm \textbf{UMOT}}_{\varepsilon}$\textbf{)}}, concerns a coupled-two-marginal formulation of the unbalanced barycenter problem where the cost function and the first marginal entropy functional are chosen as in the Hellinger-Kantorovich (HK) metric, but the second marginal entropy functional is replaced with a hard constraint --- the marginal has to match the barycenter exactly. 

Concretely, we seek to find a barycenter $\overline \nu \in \M(X)$ (a set of finite positive Radon measures on $X$, which for simplicity we take to be a compact subset of $\R^d$) of a set of $N$ measures 
\[
    \vec \mu = (\mu_1, \dots, \mu_N),\quad  \mu_i \in \M(X)
\]
with weights $\vec \lambda = (\lambda_1,\dots, \lambda_N), \lambda_i \geq 0$, adding to one, by solving {\cb a constrained coupled two-marginal problem (CC2M, see Table~\ref{tab:subscript} for more details on the subscript notation being used)}
\[
{\rm HK}_{\rm CC2M}(\vec \mu) = \inf_{\nu \in \M(X)}\left\{\sum_{i=1}^N \lambda_i \overline{\rm HK}^2(\mu_i,\nu)\right\}, 
\]
where
\[
    \overline{\rm HK}^2(\mu_i,\nu) = \inf_{\gamma \in \M(X \times X)}\Big\{ \int_{X \times X}c(x_1,x_2)d\gamma(x_1,x_2)\, + \F(\gamma_1 \mid \mu_i)\, + \overline \F(\gamma_2 \mid \nu) \Big\}.
\]
Here $\gamma_i := (\pi^i)_{\#}\gamma$ is the $i$th marginal, the cost function $c$ and entropy functional $\F$ are as is the HK metric (see Section~\ref{sec:HK} for details), but $\overline \F(\gamma_2 \mid \nu)$ vanishes if $\gamma_2 = \nu$ and is equal to $+\infty$ otherwise. 

Complimentary to \hyperref[framework2]{\textbf{(}${\rm \textbf{UMOT}}_{\varepsilon}$\textbf{)}}, we consider the unregularised problem and furthermore show that a natural \emph{least-cost} soft multi-marginal {\cb (SMM)} reformulation exists and is given by
\[
    {\rm HK}_{\rm SMM}(\vec \mu) := \inf_{\gamma \in \M(X^N)} \left\{\int_{X^N}\tilde c(\vec x)d\gamma(\vec x) + \sum_{i} \lambda_i \F(\gamma_i \mid \mu_i)\right\},
\]
where $\vec x := (x_1,\dots,x_N)$ and the least cost $\tilde c$ is given by 
\begin{equation}\label{tildec-intro}
    \tilde c(\vec x) = \inf_{x \in X} \sum_{i=1}^N \lambda_i c(x_i,x).
\end{equation}

To deepen the connection with \hyperref[framework1]{\textbf{(HKB)}}, we in fact begin the analysis by defining and studying a new conic multi-marginal problem {\cb (CMM)}, given by 
\[
    \overline{\rm HK}_{\rm CMM}(\vec \mu) = \inf_{\alpha \in \M(Y^N)} \left\{\int_{Y^N}\overline H_{\rm CMM}(\vec x, \vec s) \alpha(\vec x, \vec s)\,\mid\, (\pi^{x_i})_{\#}(s_i \alpha) = \mu_i \right\},
\] 
where $Y := X \times\, \R_+$ is the cone space (up to an equivalence relation for all \linebreak ${y = (x,s) \in Y}$ with $s=0$), and $\vec{s} = (s_1,\dots,s_{N}) \in \R_+^N$ are the mass variables. The new \emph{multi-marginal perspective cost function} $\overline H_{\rm CMM}\,\colon\, Y^N \to \overline \R$ is given by 
\[
    \overline H_{\rm CMM}(\vec x, \vec s) = \inf_{(x,s) \in Y} \inf_{t > 0}\sum_{i=1}^N \Big\{t\lambda_i c(x_i,x) + \lambda_i s_i F\left(\tfrac{t}{s_i}\right) + \lambda_i s  F\left(\tfrac{t}{s}\right)\Big\},
\]
where $F$ is the entropy function in the definition of the functional $\F$, (see Section~\ref{sec:HK}) and, as will be discussed in Section~\ref{sec:CMM}, is fundamentally different from the two-marginal perspective cost function approach from \hyperref[framework1]{\textbf{(HKB)}}.

The central result of the paper is the establishing of a three-way equality 
\[
    {\rm HK}_{\rm CC2M}(\vec \mu) = {\rm HK}_{\rm SMM}(\vec \mu) = \overline{\rm HK}_{\rm CMM}(\vec \mu).
\]
While for simplicity we only consider $X$ to be a compact subset of $\R^d$ and, to connect with \hyperref[framework1]{\textbf{(HKB)}}, we only focus on the HK setting, it is expected that the same equality should hold for the class of cost functions and entropy functionals covered by the theory of optimal entropy-transport problems developed in \cite[Part I]{LMS18}, provided that in \eqref{tildec-intro} a unique minimizer exists or the infimum is given by $+\infty$. 

The paper is organised as follows. %{\cb To help readers navigate the paper, we begin by outlining the notational convention in Section~\ref{sec:notation}.}
In Section~\ref{sec:prem} we discuss the preliminaries, which includes a reminder on the Hellinger-Kantorovich distance and its conic formulation, followed by a brief account of the theory of (unconstrained) HK barycenters \hyperref[framework1]{\textbf{(HKB)}}.

The main results of the paper are presented in Section~\ref{sec:main-results}. We begin with analysis of the new conic multi-marginal problem $\overline{\rm HK}_{\rm CMM}(\vec \mu)$ in Section~\ref{sec:CMM}, followed by showing in Section~\ref{sec:SMM} that it is equivalent to the soft multi-marginal formulation ${\rm HK}_{\rm SMM}(\vec \mu)$, for which existence of solutions is established (and hence also for the CMM formulation). Finally, in Section~\ref{sec:CC2M} we discuss the constrained coupled-two-marginal problem ${\rm HK}_{\rm CC2M}(\vec \mu)$, show that it admits a solution and establish its equivalence with the soft multi-marginal formulation. 

Section~\ref{sec:disc} is devoted to a discussion about the obtained results. In Section~\ref{sec:comp-UMOT} we discuss the relation of our results to \hyperref[framework2]{\textbf{(}${\rm \textbf{UMOT}}_{\varepsilon}$\textbf{)}}. This includes a numerical example in which we reproduce an example from \cite{BvLNS23} using our least-cost approach. This is followed by a discussion about different ways of characterising the difference between the constrained approach and \hyperref[framework1]{\textbf{(HKB)}} and includes an example about Dirac masses. Finally, in Section~\ref{sec:HKvsW-bary}, we formally discuss the general idea of trying to extend the notion of a Wasserstein barycenter to the unbalanced setting and argue that the constrained approach might be a more natural one. 

The paper finishes with Section~\ref{sec:proofs}, where proofs are gathered. 

{\cb \subsection*{Notation}\label{sec:notation}
The theory developed in this paper necessitates a heavy use of the subscript notation to concisely distinguish between different barycenter and multi-marginal problems being considered. \cref{tab:subscript} will hopefully help readers navigate this. 

\begin{table}[htbp]
  \footnotesize
  \caption{\cb The subscript notation used throughout the paper.}\label{tab:subscript}
  \begin{center} {\cb
    \begin{tabular}{|c|c|c|} \hline
     Notation & Full name  & Defined in  \\ \hline 
      ${\rm HK}_{\rm C2M}(\vec \mu)$ & \makecell{The \textbf{coupled two-marginal (C2M)} \\ Hellinger-Kantorovich barycenter problem} & \eqref{eqn-def-HK-C2M} \\ \hline
      ${\rm HK}_{\rm CMM}(\vec \mu)$ & \makecell{The \textbf{conic multi-marginal (CMM)} formulation  of \\ the Hellinger-Kantorovich barycenter problem} & \eqref{eqn-def-hk-tilde} \\ \hline 
      ${\rm HK}_{\rm CC2M}(\vec \mu)$ & \makecell{The \textbf{constrained coupled two-marginal (CC2M)}\\  Hellinger-Kantorovich barycenter problem} & \eqref{eqn-def-HK-CC2M} \\ \hline
      ${\rm HK}_{\rm SMM}(\vec \mu)$ & \makecell{The \textbf{soft multi-marginal (SMM)} formulation of \\ the constrained Hellinger-Kantorovich barycenter problem} & \eqref{eqn-def-HK-SMM} \\ \hline
      $\overline{\rm HK}_{\rm CMM}(\vec \mu)$ & \makecell{The \textbf{conic multi-marginal (CMM)} formulation of \\ the constrained Hellinger-Kantorovich barycenter problem} & \eqref{eqn-def-hk-mm} \\ \hline
    \end{tabular} }
  \end{center}
  \end{table}
}

\section{Preliminaries}\label{sec:prem}

\subsection{The Hellinger-Kantorovich distance}\label{sec:HK}
The Hellinger-Kantorovich (HK) distance between two finite positive Radon measures ${\mu_1,\mu_2 \in \M(X)}$ on a Hausdorff topological space $X$ {\cb endowed with an extended distance function \linebreak ${\d\,\colon\,X \times X \to [0,+\infty]}$}, introduced independently by three different groups \cite{CPSV16,KMV16,LMS18}, is given by 
\begin{equation}\label{eqn-def-HK}
{\rm HK}^2(\mu_1,\mu_2) = \inf_{\gamma \in \M(X \times X)}\Big\{ (c,\gamma) + \sum_{i=1,2} \F(\gamma_i \mid \mu_i) \Big\}, 
\end{equation}
where 
\[
    (c,\gamma) := \int_{X \times X} c(x_1,x_2)d\gamma(x_1,x_2)
\]
and the cost function $c$ is given by 
\begin{equation}\label{HK-cost-fct}
c(x_1,x_2) := \begin{cases}
-\log(\cos^2(\d(x_1,x_2)))\quad &\text{if }\d(x_1,x_2) < \pi/2 \\
+\infty \quad &\text{otherwise}.
\end{cases}
\end{equation}
{\cb As will become apparent in \cref{eqn-H-nice-form}, in the classical metric framework this choice of the cost function is special, as it gives rise to the conic distance over the cone space $Y = X \times \R_+$.}
The entropy functional $\F$ is given by 
\begin{equation}\label{F-func}
    \F(\gamma_i \mid \mu_i) := \int_X F(\sigma_i)d\mu_i + F_{\infty}' \gamma_i^{\perp}(X), \quad \gamma_i := (\pi^i)_{\#} \gamma, \pi^i(x_1,x_2) = x_i.
\end{equation}
In this definition we have used {\cb the }Lebesgue decomposition (see e.g. \cite[Lemma~2.3]{LMS18})
\begin{equation}\label{L-decomp}
    \gamma_i = \sigma_i \mu_i + \gamma_i^\perp, \quad \mu_i = \rho_i \gamma_i + \mu_i^\perp.
\end{equation}
and we further note that the function $F\, \colon\, [0,+\infty) \to [0,+\infty{\cb ) }$ is the relative entropy function, also known as Kullback-Leibler (KL) divergence, and is given by 
\begin{equation}\label{eqn-def-F}
    F(s) := s\log s - s  + 1,\quad \Big(\implies F_\infty':= \lim_{s \to \infty}\frac{F(s)}{s} = +\infty\;\Big),
\end{equation}
thus enforcing $\gamma _i \ll \mu_i$ for any feasible plan $\gamma$ in \eqref{eqn-def-HK}.

\subsection{Reverse formulation and lifting to the cone}\label{sec:lifting-cone}
The ${\rm HK}$ distance admits several useful reformulations, with two of them particularly relevant to the present work. Given the entropy function $F$ defined in \eqref{eqn-def-F}, the corresponding reverse entropy function $R\,\colon\, [0,\infty) \to [0,\infty]$, defined as 
\begin{equation}\label{R-ent}
    R(s) := \begin{cases}
s F(1/s)\quad &\text{if }s>0,\\
F_{\infty}' \quad &\text{if }s=0,
\end{cases}
\end{equation}
is in fact given by 
\begin{equation}\label{eqn-def-R}
    R(s) = s - \log s - 1.
\end{equation}
One can then show \cite[Theorem~3.11]{LMS18} that 
\begin{align}\label{eqn-HK-R}
    {\rm HK}^2(\mu_1,\mu_2) = \inf_{\gamma \in \M(X \times X)}\Big\{ &\int_{X \times X}c(x_1,x_2) + R(\rho_1(x_1)) + R(\rho_2(x_2))d\gamma \\
    &+\sum_{i=1,2}R_{\infty}'(\mu_i - \rho_i\gamma_i)(X)\Big\},\nonumber
\end{align}
where $R_{\infty}' = F(0) = 1$ {\cb and $\rho_1, \rho_2$ are defined through the Lebesgue decomposition \cref{L-decomp}.}

This sets the scene for the lifting to the cone space $Y := X \times\, \R_+$ (up to an equivalence relation for all $y = (x,s) \in Y$ with $s=0$), which is facilitated by the introduction of the \emph{marginal perspective cost function},
\[
    H\,\colon\, X \times \R_+ \times X \times \R_+ \to [0, +\infty],
\]  
given by
\begin{equation}\label{eqn-H-def}
    H(x_1,s_1,x_2,s_2) := \inf_{t >0 }\left\{tc(x_1,x_2) + t R\left(\frac{s_1}{t}\right) + t R\left(\frac{s_2}{t}\right)\right\}.
\end{equation}
A direct computation shows that, for $R$ given by \eqref{eqn-def-R}, $H$ admits an explicit formulation given by
\begin{equation}\label{eqn-H-nice-form}
H(x_1,s_1,x_2,s_2) = s_1 + s_2 - 2s_1^{1/2}s_2^{1/2}\exp\left(\frac{-c(x_1,x_2)}{2}\right).
\end{equation}
It can then be shown \cite[Theorem 5.8]{LMS18} that, after lifting to the cone space $Y$, we have 
\begin{equation}\label{HK-cone}
{\rm HK}^2(\mu_1,\mu_2) = \inf_{\alpha \in S(\mu_1,\mu_2)} (H, \alpha),     
\end{equation}
where 
\[
(H,\alpha) = \int_{Y \times Y}H(x_1,s_1,x_2,s_2)d\alpha(x_1,s_1,x_2,s_2)    
\]
and 
\[
S(\mu_1,\mu_2) =\left\{ \alpha \in \M(Y \times Y)\,\mid\, \h_i \alpha = \mu_i \right\},    
\]
where the homogeneous marginal $\h_i\,\colon\, \M(Y \times Y) \to \M(X)$ is defined as 
\begin{equation}\label{eqn-def-h-marg}
    \h_i \alpha := \pi^{x_i}_{\#}(s_i \alpha).
\end{equation}
Note that $\h_i \alpha = \mu_i$ ensures that the lifted counterpart to the singular part in \eqref{eqn-HK-R} satisfies 
\begin{equation}\label{R-inf-zero}
    \sum_{i=1,2}R_{\infty}'(\mu_i - \h_i \alpha)(X) = 0.
\end{equation}

\subsection{Unconstrained barycenters for the Hellinger-Kantorovich \linebreak distance}\label{sec:HK-bary}
In what follows, for simplicity, we focus on the case when $X$ is a compact subset of $\R^n$ with the distance $\d$ given by the Euclidean distance. We further suppose we are given $N$ measures
\[
    \vec{\mu} = (\mu_1,\dots,\mu_{N}), \quad \mu_i \in \M(X)
\]
and $N$ weights
\[
    \vec\lambda = (\lambda_1,\dots,\lambda_N),\quad \lambda_i \in \R_+,\quad \sum_{i=1}^N \lambda_i =1.
\]

We will now briefly recall relevant known results about the HK barycenter problem, as studied in \cite{FMS21,CP21}. We will aim to follow the presentation in \cite{FMS21}, with a few relevant notational changes.

\subsubsection{Coupled two-marginal formulation}
The \emph{unconstrained} \linebreak Hellinger-Kantorovich barycenter problem is given by 
\begin{equation}\label{eqn-def-HK-C2M}
{\rm HK}_{\rm C2M}(\vec \mu) = \inf_{\nu \in \M(X)}\left\{\sum_{i=1}^N \lambda_i {\rm HK}^2(\mu_i,\nu)\right\}, 
\end{equation}
where we have used the convention employed in \cite{FMS21} to use the subscript ${\rm C2M}$ to emphasise that this is the coupled two-marginal formulation. We refer to this formulation as \emph{unconstrained} because, in contrast to the new approach to be introduced in Section~\ref{sec:main-results}, the second marginals of optimal plans minimising ${\rm HK}^2(\mu_i,\nu)$ are not constrained to be equal to $\nu$. 

\subsubsection{Multi-marginal formulation}
In analogy with the Wasserstein distance barycenters \cite{AC11}, it is of considerable interest to study whether the HK barycenter problem defined in \eqref{eqn-def-HK-C2M} admits a multi-marginal reformulation. As established in \cite{FMS21,CP21} and further studied in \cite{BMS23}, this is possible starting from the cone formulation described in Section~\ref{sec:lifting-cone}, treating the marginal perspective cost function $H$ defined in \eqref{eqn-H-def} as a given cost function. In particular, in \cite{FMS21}, the authors consider a conic multi-marginal (CMM) formulation, which, using the notation 
\[
  X^N := \underbrace{X \times \dots \times X}_{N \text{ times}},\quad  \vec{x} = (x_1,\dots,x_{N}) \in X^N,\quad  \vec{s} = (s_1,\dots,s_{N}) \in \R_+^N,
\]
is given by 
\begin{equation}\label{eqn-def-hk-tilde}
{\rm HK}_{\rm CMM}(\vec \mu) = \inf_{\alpha \in S(\vec \mu)} (H_{\rm CMM},\alpha),
\end{equation}
where
\[
    (H_{\rm CMM},\alpha) = \int_{Y^N}H_{\rm CMM}(\vec x, \vec s) d\alpha(\vec x, \vec s),
\]
and the cost function is given by 
\begin{equation}\label{eqn-tildeH-MM}
H_{\rm CMM}(\vec x, \vec s) = \inf_{(x,s) \in Y} \sum_{i=1}^N \lambda_i H(x_i,s_i,x,s),
\end{equation}
where $H$ is given by \eqref{eqn-H-def}. The infimum in \eqref{eqn-def-hk-tilde} is taken over the set
\begin{equation}\label{eqn-def-S-vec-mu}
    S(\vec \mu) := \{ \alpha \in \M(Y^N)\,\mid\, \h_i \alpha = \mu_i\},
\end{equation}
where the homogeneous marginal $\h_i\,\colon\, \M(Y^N) \to \M(X)$ is as in \eqref{eqn-def-h-marg}, up to an obvious adjustment from $Y \times Y$ to $Y^N$, and the space $Y^N$ is defined as
\[
    Y^N := \underbrace{Y \times \dots \times Y}_{N \text{ times}}.
\]

Among other results, the following is proven. 
\begin{theorem}[\protect{\cite[Theorem~5.13]{FMS21}}]\label{thm:fms}
It holds that
    \[
        {\rm HK}_{\rm CMM}(\vec \mu) = {\rm HK}_{\rm C2M}(\vec \mu).
    \] 
\end{theorem}
The result is very difficult to obtain because the minimisation in \eqref{eqn-tildeH-MM} does not admit an explicit solution and thus a convex relaxation of the perspective cost function ${H}_{\rm CMM}$ needs to be introduced and the associated dual problems considered. Furthermore, it is shown in \cite[Section 8]{FMS21}, that {\cb the unconstrained Hellinger-Kantorovich barycenter problem ${\rm HK}_{\rm C2M}(\vec \mu)$ does not admit an original space -based multi-marginal formulation (the so-called soft multi-marginal formulation). This} 
%the problem cannot be mapped back to the original space to a soft multi-marginal formulation, which
is a significant deviation from the Wasserstein barycenters (see Section~\ref{sec:HKvsW-bary} for an extended discussion on this). It is thus of considerable interest to explore what, ideally minimal, changes are needed to ``simplify'' the framework developed in \cite{FMS21,CP21,BMS23} and e.g. relate it to the recently explored idea of unbalanced multi-marginal optimal transport \cite{BvLNS23} (see Section~\ref{sec:comp-UMOT} for an extended discussion). This is what we will discuss now. 

\section{Main results}\label{sec:main-results}
\subsection{Conic multi-marginal formulation}\label{sec:CMM}
The starting point of the present work is a new conic multi-marginal formulation of, what will turn out to be, the \emph{constrained} HK barycenter problem. In contrast to \eqref{eqn-def-hk-tilde}, we introduce it as
\begin{equation}\label{eqn-def-hk-mm}
\overline{\rm HK}_{\rm CMM}(\vec \mu) := \inf_{\alpha \in S(\vec \mu)} (\overline H_{\rm CMM}, \alpha),
\end{equation}
where the new \emph{multi-marginal perspective cost function} is given by
\begin{equation}\label{eqn-def-H-MM}
\overline H_{\rm CMM}(\vec x, \vec s) := \inf_{(x,s) \in Y} \inf_{t > 0}\sum_{i=1}^N \Big\{t\big(\lambda_i c(x_i,x) + \lambda_i R\left(\tfrac{s_i}{t}\right) + \lambda_i R\left(\tfrac{s}{t}\right)\big)\Big\}
\end{equation}
and the set $S(\vec \mu)$ is as in \eqref{eqn-def-S-vec-mu}.

The initial idea behind defining such a cost was to look back at the two-marginal formulation and, using the reverse formulation from \eqref{eqn-HK-R}, rewrite it, as much as possible, as a collection of integrals with respect to two-marginal plans ${\gamma_i \in \M(X \times X)}$, namely 
\begin{align*}
\cb {\rm HK}_{\rm C2M}(\vec \mu) = \inf_{\nu \in \M(X)} \sum_{i=1}^N &\cb \inf_{\gamma_i \in \M(X \times X)} \Big( \lambda_i(\mu_i - \rho_{i,1}\gamma_{i,1})(X) + \lambda_i (\nu - \rho_{i,2}\gamma_{i,2})(X) \\ 
&\cb + \int_{X \times X} \lambda_i c(x_1,x_2) + \lambda_i R(\rho_{i,1}(x_1) + \lambda_i R(\rho_{i,2}(x_2))) d\gamma_i\Big)%\\
%    &+ \lambda_i(\mu_i - \rho_{i,1}\gamma_{i,1})(X) + \lambda_i (\nu - \rho_{i,2}\gamma_{i,2})(X)
\end{align*}
where, similarly to \eqref{L-decomp}, we have used the Lebesgue decomposition
\begin{subequations}\label{L-decomp-multi}
    \begin{alignat}{3}
        \gamma_{i,1} &= \sigma_{i,1} \mu_i + \gamma_{i,1}^\perp, &&\quad\quad \mu_i &= \rho_{i,1} \gamma_{i,1} + \mu_i^\perp\\
        \gamma_{i,2} &= \sigma_{i,2}\nu + \gamma_{i,2}^{\perp}, &&\quad\quad \nu &= \rho_{i,2} \gamma_{i,2} + \nu^\perp.
    \end{alignat}
\end{subequations}
\begin{remark}[Potential notational confusion]
    To avoid notation becoming too cumbersome, in the coupled-two-marginal formulations, we use $\gamma_i \in \M(X \times X)$ to denote the $i$th two-marginal plan and $\gamma_{i,j} \in \M(X)$ to denote the $j$th marginal of the $i$th plan. In contrast, in the multi-marginal formulation, we have a multi-marginal plan $\gamma \in M(X^N)$ and $\gamma_i \in M(X)$ denotes its $i$th marginal. This can be particularly confusing when $N=2$ and we will try to always indicate the space to which $\gamma_i$ belongs to minimise potential confusion.
\end{remark}

As first attempt, by trying to mimic the Wasserstein barycenters theory, a candidate for a multi-marginal cost function is given by
\[
\inf_{(x,s) \in Y} \sum_{i=1}^N \Big\{\big(\lambda_i c(x_i,x) + \lambda_i R\left(s_i\right) + \lambda_i R\left(s\right)\big)\Big\}.
\]
Secondly, to mimic the lifting to the cone strategy developed in \cite{LMS18} and recalled in Section~\ref{sec:lifting-cone}, to arrive at a multi-marginal perspective cost function, we further introduce an extra infimization over $t>0$, 
\[
\overline H_{\rm CMM}(\vec x, \vec s) = \inf_{(x,s) \in Y} \inf_{t > 0}\sum_{i=1}^N \Big\{t\big(\lambda_i c(x_i,x) + \lambda_i R\left(\tfrac{s_i}{t}\right) + \lambda_i R\left(\tfrac{s}{t}\right)\big)\Big\}.
\]

To directly compare with the cost function $H_{\rm CMM}$ considered in \cite{FMS21}, we recall that it is given by
\[
H_{\rm CMM}(\vec x, \vec s) = \inf_{(x,s) \in Y} \sum_{i=1}^N \inf_{t_i > 0}\Big\{t_i\big(\lambda_i c(x_i,x) + \lambda_i R\left(\tfrac{s_i}{t_i}\right) + \lambda_i R\left(\tfrac{s}{t_i}\right)\big)\Big\}.
\]
The difference is thus that in the known multi-marginal formulation, the extra infimization over $t_i >0$ is introduced separately for each two-marginal problem. The key distinction is that in the new approach, we ``exchange'' the order of taking the infimum over $t_i > 0$ and the summation in $i$ and as a result only infimize over a joint multi-marginal $t>0$. The idea of a multi-marginal perspective function is adapted from \cite{BD23}, where a similar notion was used to study the entropic regularisation of the unbalanced optimal transport problems.

We begin by proving the following.
\begin{lemma}\label{lem-1}
The multi-marginal perspective cost function $\overline H_{\rm CMM}$ defined in \eqref{eqn-def-H-MM} can be equivalently written as
\begin{align*}
\overline H_{\rm CMM}(\vec x, \vec s) &= \inf_{(x,s) \in Y} \left( \sum_{i=1}^N \lambda_i s_i + s - 2s^{\tfrac12}\left(\prod_{{\cb j}=1}^N s_{\cb j}^{\tfrac{\lambda_{\cb j}}{2}}\right)\exp\left(\frac{-\sum_{{\cb k}=1}^N \lambda_{\cb k} c(x_{\cb k},x)}{2}\right)\right)\\
&= \inf_{s \in \R_+} \left( \sum_{i=1}^N \lambda_i s_i + s - 2s^{\tfrac12}\left(\prod_{{\cb j}=1}^N s_{\cb j}^{\tfrac{\lambda_{\cb j}}{2}}\right)\exp\left(\cb -\frac12\inf_{x \in X}\sum_{k=1}^N \lambda_k c(x_k,x)\right)\right)\\
&= \sum_{i=1}^N \lambda_i s_i - \prod_{{\cb j}=1}^N s_{\cb j}^{\lambda_{\cb j}}\exp\left(-\inf_{x \in X} \sum_{{\cb k}=1}^N \lambda_{\cb k} c(x_{\cb k},x)\right).
\end{align*}
\end{lemma}
For proof, see Section~\ref{sec-proof-lemma}. The key point here is that the minimizer in $s$ has an explicit formula, hence allowing us to obtain the last equality.

\subsection{Soft multi-marginal formulation}\label{sec:SMM}

The result of Lemma~\ref{lem-1} allows us to characterise the new conic multi-marginal problem as a soft multi-marginal (SMM) problem on the space $X^N$, which is given by
\begin{equation}\label{eqn-def-HK-SMM}
    {\rm HK}_{\rm SMM}(\vec \mu) := \inf_{\gamma \in \M(X^N)} \left\{(\tilde c, \gamma) + \sum_{i} \lambda_i \F(\gamma_i \mid \mu_i)\right\},
\end{equation}
where the cost function $\tilde c\,\colon\, X^N \to \overline \R$ is given by
\begin{equation}\label{eqn-def-c-tilde}
\tilde c(\vec x) = \inf_{x \in X} \sum_{i=1}^N \lambda_i c(x_i,x).
\end{equation}

We first prove that the soft multi-marginal problem has a solution.
\begin{proposition}\label{prop-SMM-sol}
    There exists $\bar \gamma \in \M(X^N)$ such that 
    \[
    {\rm HK}_{\rm SMM}(\vec \mu) = (\tilde c, \bar\gamma) + \sum_{i} \lambda_i \F(\bar\gamma_i \mid \mu_i).
    \]
\end{proposition}
For proof, see Section~\ref{sec-proof-SMM-sol}. The result follows from simple adjustments to the argument in \cite[Theorem~3.3]{LMS18}, which concerns the case $N=2$. 

This brings us to the first main result of the paper. 
\begin{theorem}\label{thm-1}
The {\cb conic }multi-marginal formulation of the barycenter problem \linebreak  $\overline{\rm HK}_{\rm CMM}(\vec \mu)$ defined in \eqref{eqn-def-hk-mm} admits a soft multi-marginal reformulation given by \eqref{eqn-def-HK-SMM}. In other words,
\[
\overline{\rm HK}_{\rm CMM}(\vec \mu) = {\rm HK}_{\rm SMM}(\vec \mu). 
\]
Furthermore, if $\bar \gamma \in \M(X^N)$ is a solution to ${\rm HK}_{\rm SMM}(\vec \mu)$, then $\bar \alpha \in \M(Y^N)$ given by 
\[
    \bar \alpha := (x_1,\bar \rho_1(x_1),\dots,x_N, \rho_N(x_N))_{\#}\bar \gamma,
\]
where 
\[
    \mu_i = \bar \rho_i \bar \gamma_i + \mu_i^{\perp},
\]
is a solution to $\overline{\rm HK}_{\rm CMM}(\vec \mu)$.
\end{theorem}
For proof, see Section~\ref{sec-proof-thm-1}.\\

\subsection{Constrained coupled two-marginal formulation}\label{sec:CC2M}
It is perhaps tempting to think that the new {\cb conic} multi-marginal formulation {\cb $\overline {\rm HK}_{\rm CMM}(\vec \mu)$ defined in \eqref{eqn-def-hk-mm}} simply coincides with the unconstrained coupled two-marginal {\cb barycenter problem ${\rm HK}_{\rm C2M}(\vec \mu)$} introduced in \eqref{eqn-def-HK-C2M}{\cb . This cannot be the case, however, since, on the one hand, through Theorem~\ref{thm-1} we know that $\overline {\rm HK}_{\rm CMM}(\vec \mu)$ is equivalent to the soft multi-marginal reformulation ${\rm HK}_{\rm SMM}(\vec \mu)$, and on the other, through \cite[Section~8]{FMS21}, we know that the unconstrained problem ${\rm HK}_{\rm C2M}(\vec \mu)$ is not equivalent to it.}
% but, as was shown in \cite[Section~8]{FMS21}, the result of Theorem~\ref{thm-1} ensures this cannot be the case.
It turns out, however, that the difference {\cb between $\overline {\rm HK}_{\rm CMM}(\vec \mu)$ and ${\rm HK}_{\rm C2M}(\vec \mu)$} can be described quite explicitly, namely it lies in introducing an extra constraint in the coupled two-marginal formulation. We define a \emph{constrained coupled two-marginal} (CC2M) HK barycenter problem, given by
\begin{equation}\label{eqn-def-HK-CC2M}
    {\rm HK}_{\rm CC2M}(\vec \mu) := \inf_{\nu \in \M(X)}\left\{\sum_{i=1}^N \lambda_i \overline{\rm HK}^2(\mu_i,\nu)\right\}, 
\end{equation}
where, in contrast to the usual formulation of the HK distance defined in \eqref{eqn-def-HK}, we have 
\[
    \overline{\rm HK}^2(\mu_i,\nu) := \inf_{\gamma \in \M(X \times X)}\Big\{ (c,\gamma) + \F(\gamma_1 \mid \mu_i) + \overline \F(\gamma_2 \mid \nu) \Big\},
\]
{\cb and, reusing the general entropy functional notation \eqref{F-func},}
\begin{equation}\label{eqn-bar-F}
    \overline \F(\gamma_i \mid \mu_i) := \int_X \overline F(\sigma_i)d\mu_i + \overline F_{\infty}' \gamma_i^{\perp}(X), \quad \overline F(s) = \begin{cases} 0,\;&\text{ if } s = 1,\\ +\infty,\;&\text{ otherwise.} \end{cases}
\end{equation}
In other words, {\cb $\overline F$ is the convex indicator $\iota_{\{1\}}$ and thus} we impose a hard constraint on the second marginal of $\gamma$ to exactly coincide with $\nu$. This of course implies that $\overline{\rm HK}$ is no longer a metric, but nonetheless retains some features of the ${\rm HK}$ distance, as we still use the same cost function and the entropy functional for the first marginal remains the same. In particular, the one-sided hard-constraint setup is covered by the general theory of optimal entropy-transport problems developed in \cite[Part I]{LMS18} -- see Example E.8 in Section~3.3 therein. As a result, we can show that ${\rm HK}_{\rm CC2M}(\vec \mu)$ admits at least one solution.
\begin{proposition}\label{prop-CC2M-sol}
    There exists $\bar \nu \in \M(X)$ such that
    \[
        {\rm HK}_{\rm CC2M}(\vec \mu) = \sum_{i=1}^N \lambda_i \overline{\rm HK}^2(\mu_i,\bar \nu).
    \]
\end{proposition}
For proof, see Section~\ref{sec:proof-CC2M-sol}.

We further note that the idea of a one-sided hard constraint was pioneered in \cite{BvLNS23}, albeit in a slightly different context (see Section~\ref{sec:comp-UMOT} for an extended discussion). 

This brings us to the second main result of the paper. 
\begin{theorem}\label{thm-2}
    The constrained coupled two-marginal formulation of the barycenter problem ${\rm HK}_{\rm CC2M}(\vec \mu)$ defined in \eqref{eqn-def-HK-CC2M} admits a soft multi-marginal reformulation given by \eqref{eqn-def-HK-SMM}. In other words, 
    \[
        {\rm HK}_{\rm CC2M}(\vec \mu) = {\rm HK}_{\rm SMM}(\vec \mu).
    \]
    {\cb Furthermore, if $\bar \gamma \in \M(X^N)$ is a solution to ${\rm HK}_{\rm SMM}(\vec \mu)$, then either \\
    (i) $\bar \gamma$ is the zero measure on $X^N$, in which case the zero measure on $X$ is a solution to ${\rm HK}_{\rm CC2M}(\vec \mu)$; or \\
     (ii)~$\bar \gamma$ is not the zero measure on $X^N$, in which case $\bar \nu = T_{\#}\bar \gamma$ is a solution to ${\rm HK}_{\rm CC2M}(\vec \mu)$. Here $T\,\colon\,A \to X$ is a mapping from 
    \[
    A := \left\{ \vec x = (x_1,\dots,x_N) \in X^N \mid \bigcap_{i=1}^N B_{\frac{\pi}{2}(x_i)} \neq \emptyset\right\}
    \]
    to the unique solution to 
    \[
        \inf_{x \in X} \sum_{i=1}^N \lambda_i c(x_i,x),
    \]
    which exists for all $\vec x \in A$.     
    }
\end{theorem}
For proof, see Section~\ref{sec-proof-thm-2}.

We note that by combining Theorem~\ref{thm-1} and Theorem~\ref{thm-2} we thus obtain a three-way equality
\[
    {\rm HK}_{\rm CC2M}(\vec \mu) = \overline {\rm HK}_{\rm CMM}(\vec \mu) = {\rm HK}_{\rm SMM}(\vec \mu),
\]
as opposed to the two-way equality for the unconstrained HK barycenter framework for which we have
\[
    {\rm HK}_{\rm C2M}(\vec \mu) = {\rm HK}_{\rm CMM}(\vec \mu).
\]
See Section~\ref{sec:comp-uncon} for an extended discussion on how the two approaches differ.

\section{Discussion}\label{sec:disc}
Before we proceed to the details of proofs, we finish the main part of the paper by discussing how problems ${\rm HK}_{\rm CC2M}(\vec \mu)$ and ${\rm HK}_{\rm SMM}(\vec \mu)$ relate to the recently studied theory of unbalanced multi-marginal optimal transport \cite{BvLNS23} and also how the constrained framework compares with the unconstrained one. We finally also informally examine whether the constrained setup can be seen as a more natural extension of the Wasserstein barycenter problem to the unbalanced setting. 

\subsection{Relation to the unbalanced multi-marginal optimal transport theory}\label{sec:comp-UMOT}

A result complimentary to Theorem~\ref{thm-2} is proven in \cite[Theorem~5.2]{BvLNS23} in the case when both problems are additionally entropy-regularised and for a much wider class of cost functions and entropy functionals.  The key difference is that (a non-regularised counterpart to) the soft multi-marginal formulation considered therein is given by {\cb 
\begin{equation}\label{eqn-hat-HK-SMM}
    \widehat{\rm HK}_{\rm SMM}(\vec \mu) := \inf_{\gamma \in \M(X^N \times X)} (\hat c,\gamma) + \sum_{i=1}^N \lambda_i \F (\gamma_i \mid \mu_i), %\quad \hat c(\vec x, x_{N+1}) := \sum_{i=1}^N \lambda_i c(x_i,x_{N+1}), 
\end{equation}
where
\[
   \hat c(\vec x, x_{N+1}) := \sum_{i=1}^N \lambda_i c(x_i,x_{N+1}), 
\]}
meaning that an extra $X$ space is introduced and the cost function is not defined as an infimum. Putting aside entropic regularisation, formally the non-regularised version of their result is that in fact 
\[
    \widehat{\rm HK}_{\rm SMM}(\vec \mu) = {\rm HK}_{\rm CC2M}(\vec \mu). 
\]

In \cite[Remark~5.5]{BvLNS23} and closing remarks the authors identify a result as in Theorem~\ref{thm-2} as something desirable{\cb , as it would establish sparsity of the optimal plan in \eqref{eqn-hat-HK-SMM} in its last component. {\cb More broadly, such a result fits with wider literature on studying the sparsity of optimal plans in multi-marginal problems \cite{gangbo1998optimal,pass2012local,Pass14,brizzi2024pwassersteinbarycenters,brizzi2024hwassersteinbarycenters}. }} We show that, at least in the HK-related setting, it can be done. Our proof, which will be presented in Section~\ref{sec-proof-thm-2}, relies merely on using the reverse formulation to obtain the inequality ${\rm HK}_{\rm CC2M}(\vec \mu) \geq {\rm HK}_{\rm SMM}(\vec \mu)$ and, at least on the face of it, should apply to a wide class of entropy functions and cost functions (provided that the infimization in \eqref{eqn-def-c-tilde} admits a unique minimizer, or the infimum is equal to $+\infty$). %  {\'S}wi{\k{e}}ch Swiech \k{a} %  Swiech \k{a} %{\'S}wiech \k{a}  %{\cb More broadly, such a result fits with wider literature on studying the sparsity of the optimal plans in multi-marginal problems \cite{gangbo1998optimal,pass2012local,Pass14,brizzi2024hwassersteinbarycenters,brizzi2024pwassersteinbarycenters}.} % it that in the extended space $X^{N+1}$ formulation the extra dimension $X$ is essentially redundant and points to sparsity of the optimal multi-marginal transport plan.

We further note that our new conic multi-marginal formulation provides an explicit link between the results of \cite{BvLNS23} and the work on HK barycenters in \cite{FMS21,CP21}, which, to the best of our knowledge, is new. %{\cb More broadly, such a result fits with wider literature on studying the sparsity of the optimal plans in multi-marginal problems.}

\begin{example}
    To showcase that the least-cost approach works in practice, we {\cb reuse the setup of } the example about computing the barycenters of 1D unbalanced Gaussians presented in \cite[Section 6.1]{BvLNS23}. We consider two empirical measures $\mu_1$ and $\mu_2$ obtained from sampling truncated normal distributions $\mathcal{N}(0.2, 0.05)$ and $2\mathcal{N}(0.8,0.08)$ on $[0,1]$ using a uniform grid with 200 points. {\cb To showcase the fact that the least-cost unbalanced approach also makes sense for other cost functions, }
    %In this case the support of measures is comfortably contained within $\bigcap_{i=1}^2 B_{\frac{\pi}{2}}({\rm supp}\mu_i)$, and thus, for all practical purposes, we can use the HK cost from \eqref{HK-cost-fct} interchangeably with the quadratic cost $c_{\rm quad}(x,y) := |x-y|^2$. To confirm this intuition,
    we use both the quadratic cost and the HK cost in our tests. 

    To approximately solve the soft multi-marginal problem ${\rm HK}_{\rm SMM}(\vec \mu)$, we use the unbalanced Sinkhorn algorithm \cite{CPSV16}, as implemented in the \textsc{POT} (Python Optimal Transport) toolbox \cite{POT}. We publish a Jupyter notebook detailing the computation\footnote{\url{https://github.com/mbuze/CHK_barycenters}}.
    
    The results are presented in Figure~\ref{fig1} and clearly {\cb the barycenters computed are qualitatively comparable the ones shown in} 
    %we are able to very closely reproduce
    the right column of \cite[Figure~2]{BvLNS23}.
\end{example}

\begin{figure}[htbp]
    \hfill
    \begin{subfigure}[c]{0.99\textwidth}
        \centering
                {\includegraphics[trim = 0 0 0 0,clip,width=\textwidth]{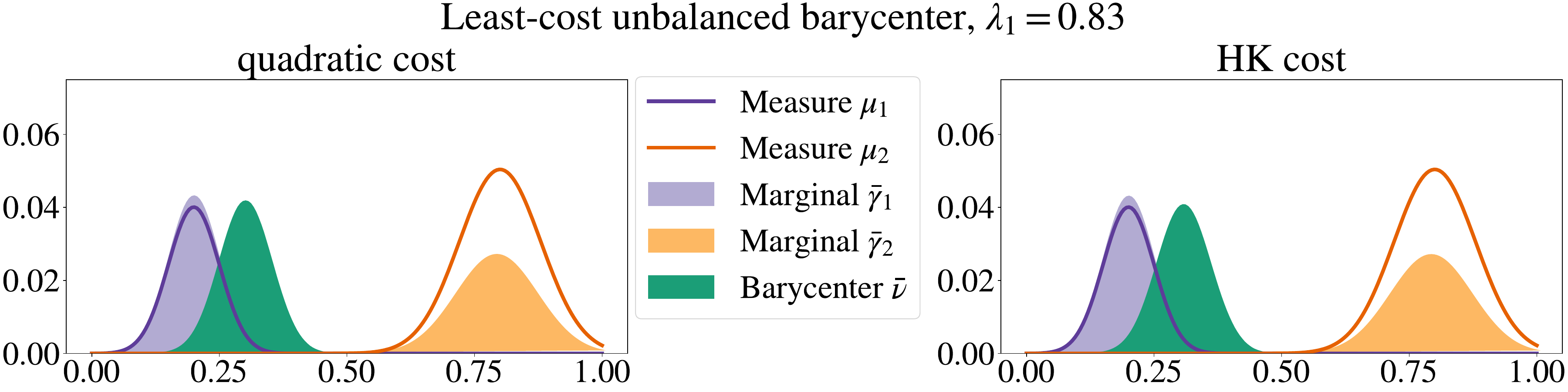}}
    \end{subfigure}
    \hfill\\[10pt]
    \begin{subfigure}[c]{0.99\textwidth}
        \centering
                {\includegraphics[trim = 0 0 0 0,clip,width=\textwidth]{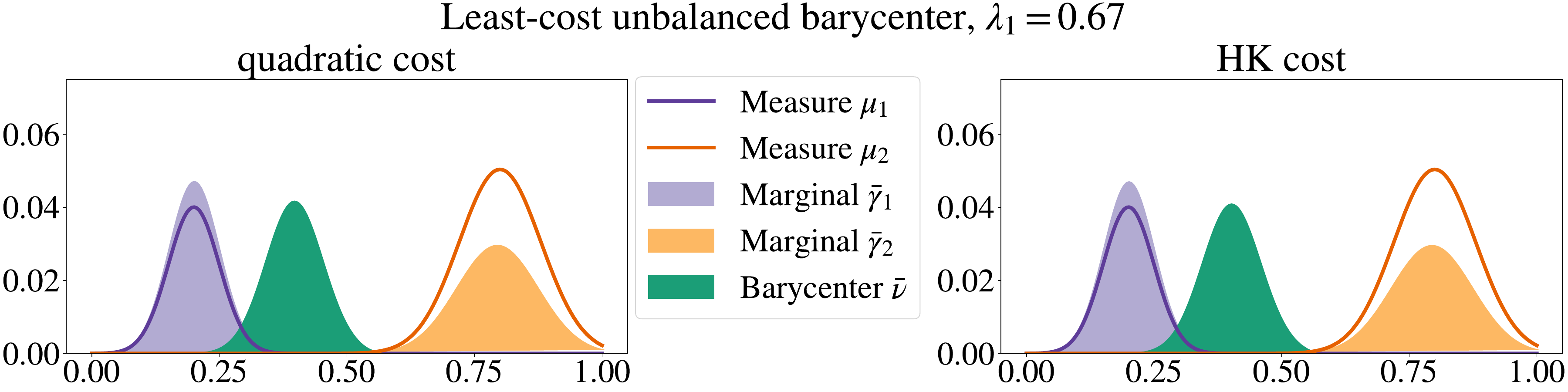}}
    \end{subfigure}
    \hfill\\[10pt]
    \begin{subfigure}[c]{0.99\textwidth}
        \centering
                {\includegraphics[trim = 0 0 0 0,clip,width=\textwidth]{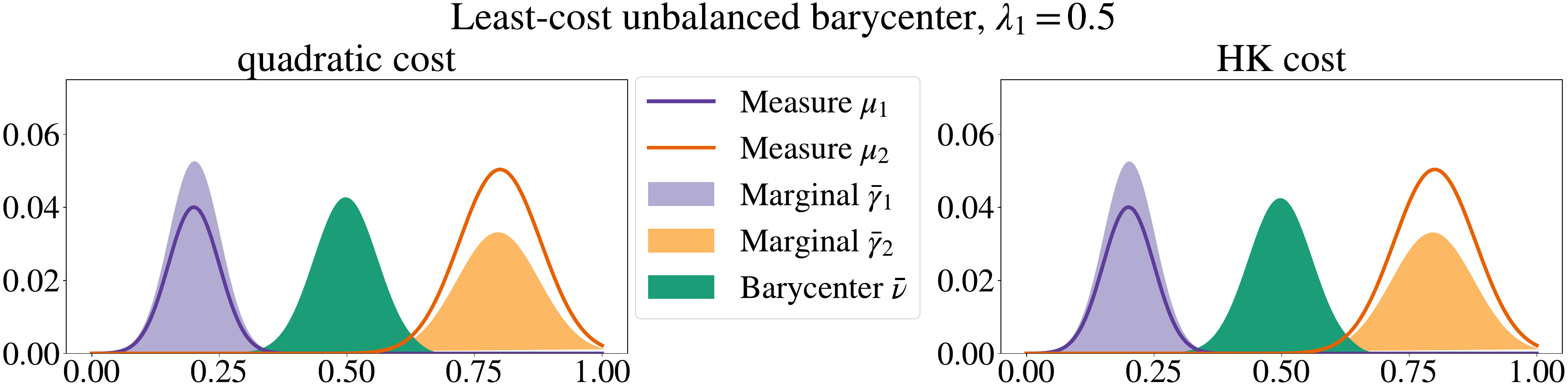}}
    \end{subfigure}
    \hfill\\[10pt]
    \begin{subfigure}[c]{0.99\textwidth}
        \centering
                {\includegraphics[trim = 0 0 0 0,clip,width=\textwidth]{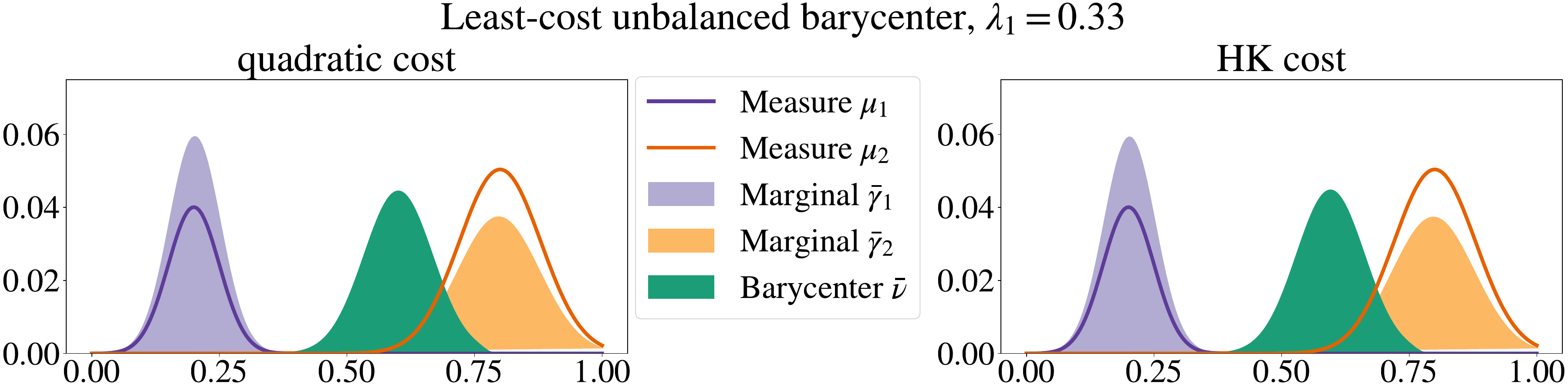}}
    \end{subfigure}
    \hfill\\[10pt]
    \begin{subfigure}[c]{0.99\textwidth}
        \centering
                {\includegraphics[trim = 0 0 0 0,clip,width=\textwidth]{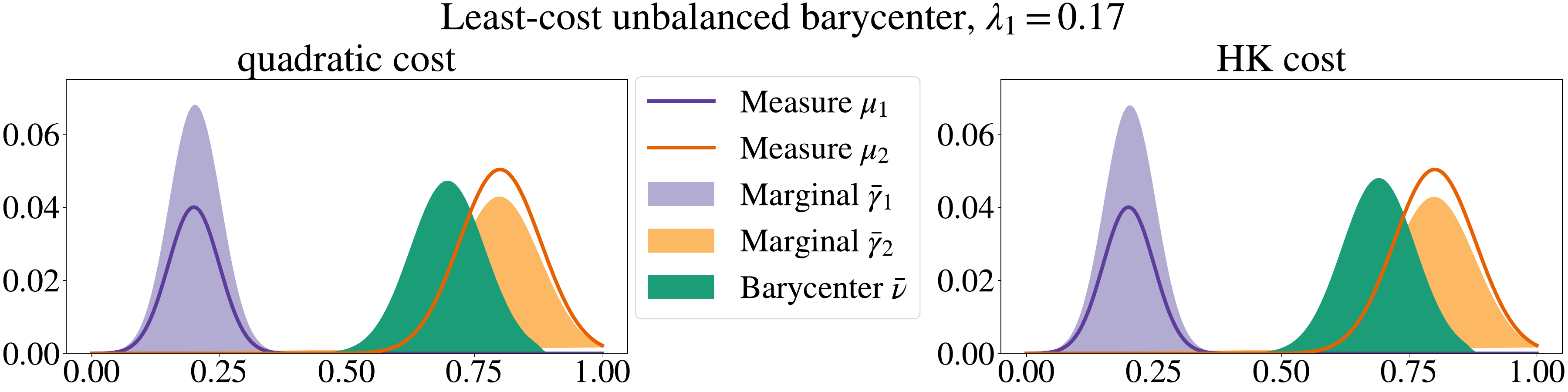}}
    \end{subfigure}
    \hfill
    \caption{Computation of the least-cost unbalanced barycenter between measures $\mu_1$ and $\mu_2$ with weights $\lambda_1$ and $\lambda_2 = 1-\lambda_1$, obtained by approximately solving ${\rm HK}_{\rm SMM}(\vec \mu)$ defined in \eqref{eqn-def-HK-SMM} using quadratic cost (left column) and HK cost (right column). We plot the input measures, the marginals of the optimal plan $\bar \gamma$ and the resulting barycenter $\bar \nu$. %Note that the difference between the quadratic cost and HK cost is negligible.
    } 
       \label{fig1}
\end{figure}

\subsection{Comparison with the unconstrained setting}\label{sec:comp-uncon}
In this section we will gather several observations relating the constrained approach proposed in this paper with previous work on the unconstrained HK barycenters in \cite{FMS21,CP21,BMS23}. 

We start with presenting two ways in which the difference between the two approaches can be described. 
\begin{proposition}\label{prop-C2MvsCC2M}
    It holds that  
    \[
    {\rm HK}_{\rm C2M}(\vec \mu) \leq {\rm HK}_{\rm CC2M}(\vec \mu).    
    \]
    The {\cb equality can only be achieved }
    %two formulations can only coincide
    if the measures $\vec \mu$ are such there exists minimiser $\bar \nu$ of ${\rm HK}_{\rm C2M}(\vec \mu)$ for which there exist minimizers $\bar \gamma^i$ of ${\rm HK}^2(\mu_i, \bar \nu)$ for which $\bar \gamma^i_{2} = \bar \nu$ for each $i$.
\end{proposition}
\begin{proof}
    This is an immediate consequence of the fact that any minimiser of \linebreak ${\rm HK}_{\rm CC2M}(\vec \mu)$ is a competitor in the unconstrained problem and further that \linebreak {\cb ${F(s) \geq F(1) = 0}$ for any $s \in [0,+\infty)]$.}
\end{proof}

A somewhat less trivial characterisation of the difference between the two approaches follows from the dual description of the conic cost functions. 
\begin{lemma}
    The joint multi-marginal perspective cost function $\overline H_{\rm CMM}$ defined \eqref{eqn-def-H-MM} admits a dual description 
    \begin{align}\label{eqn-tildeH-dual}
        \overline H_{\rm CMM}(\vec x, \vec s) = \inf_{x,s \in Y} \sup_{\phi_1,\dots,\phi_N,\phi_{N+1}} \Big\{& \sum_{i=1}^N -s_i(\lambda_i F)^*(-\phi_i) - s F^*(-\phi_{N+1})\\ & \Bigm| \sum_{i=1}^N \phi_i + \phi_{N+1} \leq \sum_{i=1}^N \lambda_i c(x_i,x). \Big\},\nonumber
    \end{align}
    where $\phi_i \in \R_+$. The dual description of the unconstrained multi-marginal cost function $H_{\rm CMM}$ from \eqref{eqn-tildeH-MM} is on the other hand given by 
    \begin{align*}
    H_{\rm CMM}(\vec x, \vec s) = \inf_{x,s \in Y} \sum_{i=1}^N \sup_{\phi_{i,1}, \phi_{i,2}} \Big\{& -s_i(\lambda_i F)^*(-\phi_{i,1}) - s(\lambda_i F)^*(-\phi_{i,2})\\ &\Bigm| \phi_{i,1} + \phi_{i,2} \leq \lambda_i c(x_i,x) \Big\},
    \end{align*}
    where $\phi_{i,j} \in \R_+$. 
\end{lemma}

Based on this result, we can establish the following. 
\begin{proposition}
  It holds that
  \[
    H_{\rm CMM}(\vec x, \vec s) \leq \overline H_{\rm CMM}(\vec x, \vec s),
  \]
  which, as already established in Proposition~\ref{prop-C2MvsCC2M}, implies that 
  \[
    {\rm HK}_{\rm CMM}(\vec \mu) \leq \overline{\rm HK}_{\rm CMM}(\vec \mu).
  \]
  %The two formulation can coincide only when there exists a collection $\bar \phi_{1}, \dots \bar \phi_{N+1}$ maximising \eqref{eqn-tildeH-dual} that satisfies $\bar \phi_{i} + \lambda_i \bar \phi_{N+1} \leq \lambda_i c(x_i,x)$ for each $i$.
\end{proposition}
\begin{proof}
    Fix $(x,s) \in Y$ and suppose $(\bar \phi_{i,1},\bar \phi_{i,2}) \in \R^2$ maximizes 
    \[
        \sup_{\phi_{i,1}, \phi_{i,2}} \Big\{ -s_i(\lambda_i F)^*(-\phi_{i,1}) - s(\lambda_i F)^*(-\phi_{i,2}) \Bigm| \phi_{i,1} + \phi_{i,2} \leq \lambda_i c(x_i,x) \Big\}.
    \]
    We can then set $\hat \phi_i := \bar \phi_{i,1}$ and $\hat \phi^j_{N+1} := \frac{\bar \phi_{j,2}}{\lambda_j}$ for some $j \in \{1,\dots,N\}$, which defines a competitor in \eqref{eqn-tildeH-dual} and thus
    \begin{align*}
        \overline H_{\rm CMM}(\vec x, \vec s) &\geq \inf_{(x,s) \in Y} \sum_{i=1}^N - s_i(\lambda_i F)^*(\hat \phi_i) - sF^*(-\hat \phi^j_{N+1})\\
        &= \inf_{(x,s) \in Y} \sum_{i=1}^N -s_i(\lambda_i F)^*(\bar \phi_{i,1}) - s F^*(-\bar \phi_{j,2}/\lambda_j)
    \end{align*}
    Since this holds for any particular choice of $j$, we can take a convex combination with weights $\vec \lambda$ of the right-hand side and arrive at 
    \begin{align*}
        \overline H_{\rm CMM}(\vec x, \vec s) &\geq \inf_{(x,s) \in Y} \sum_{j=1}^N \lambda_j \left(\sum_{i=1}^N -s_i(\lambda_i F)^*(\bar \phi_{i,1}) - s F^*(-\bar \phi_{j,2}/\lambda_j)\right)\\
        &= \inf_{(x,s) \in Y} \sum_{i=1}^N -s_i(\lambda_i F)^*(\bar \phi_{i,1}) + \sum_{j=1}^N s \lambda_j F^*(-\bar \phi_{j,2} / \lambda_j)\\
        &= H_{\rm CMM}(\vec x, \vec s).
    \end{align*}
    The last equality follows from the fact that{\cb , for any $j$,}
    \[
        {\cb -s(\lambda_j F)^*(-\phi_{j,2}) = -s \lambda_j F^*\left(\frac{-\phi_{j,2}}{\lambda_j}\right),}
    \]
    {\cb which is the standard scaling property of the Legendre transform.}
\end{proof}

We further discuss an explicit example to emphasise that, at least for Dirac masses, the constrained approach differs from the unconstrained one quite considerably. 

\begin{example}[Dirac masses]
    Consider $N$ Dirac masses $\mu_i = \delta_{z_i}$ for $z_i \in X$ and we recall that we assume that $X$ is (a compact subset of) $\R^n$. It readily follows from the analysis of the cost function $\tilde c$ in Lemma~\ref{lem-c-tilde} that if
    \begin{equation}\label{eqn-inter-balls}
        \bigcap_{i=1}^N B_{\frac{\pi}{2}}(z_i) \neq \emptyset,
    \end{equation}
    then the barycenter $\bar \nu$ {\cb obtained by minimizing in ${\rm HK}_{\rm CC2M}(\vec \mu)$ as in Proposition~\ref{prop-CC2M-sol},} is given by a single Dirac mass
    \[
        \bar\nu = M \delta_{T(\vec z)},
    \]
    where $T(\vec z)$ is the mapping
    \[
        \vec z \mapsto \min_{z \in X}\sum_{i=1}^N \lambda_i c(z_i,z)
    \]
    which is shown to exist in Lemma~\ref{lem-c-tilde}, provided that \eqref{eqn-inter-balls} holds. Furthermore, an explicit calculation reveals that the mass $M$ is given by the formula
    \begin{equation}\label{eqn-bary-mass-formula}
        M = \prod_{i=1}^N m_i^{\lambda_i} \exp(-\tilde c(\vec z)) = \prod_{i=1}^N m_i^{\lambda_i} \exp\left(- \sum_{{\cb j}=1}^N \lambda_{\cb j} c(z_{\cb j},T(\vec z))\right),
    \end{equation}
    which can be seen as a variant of a weighted geometric mean for masses of input measures. 

    If, on the other hand, \eqref{eqn-inter-balls} does not hold, then the barycenter is given by the zero measure. On an informal level this is consistent with \eqref{eqn-bary-mass-formula}, since 
    \[
        M \to 0,\;\text{ as }\; \tilde c(\vec z) \to +\infty.
    \]

    This example reveals that, for small distances between source points, the constrained HK barycenter behaves similarly to the Wasserstein barycenter and degenerates if the distance between the support of at least one of the input measures from the supports of other measures is too large. As thoroughly discussed in \cite{BMS23}, a radically different, clustering behaviour can occur in the unconstrained setup. At the same time, it is feasible to imagine a (perhaps not very efficient) clustering algorithm based on the constrained approach, since we can interpret the barycenter being the zero measure as an indication that there are at least two clusters in the data set and we can tune what we consider a cluster by adjusting the length over which transport is preferred over mass creation/annihilation. 
\end{example}

\subsection{Extending Wasserstein barycenters to the unbalanced setting}\label{sec:HKvsW-bary}
Using the hard-constraint entropy functional defined in \eqref{eqn-bar-F}, the multi-marginal formulation of the Wasserstein barycenter problem, as studied in \cite{AC11}, can be formulated as 
\[
    {\rm W}_{\rm MM}(\vec \mu)= \inf_{\gamma \in \M(X^N)}\left\{(\tilde c,\gamma) + \sum_{i=1}^N \lambda_i \overline \F(\gamma_i \mid \mu_i)\right\}, \quad \tilde c(\vec x) := \inf_{x \in X}\sum_{i=1}^N \lambda_i c(x_i,x), 
\]
where now the masses of each $\mu_i$ have to coincide, and, canonically, we have \linebreak ${c(x,y) = |x-y|^2}$. By comparing with the definition of ${\rm HK}_{\rm SMM}(\vec \mu)$ in \eqref{eqn-def-HK-SMM}, the soft multi-marginal formulation is a very intuitive extension of the Wasserstein barycenter problem to the unbalanced setting -- we replace $N$ hard marginal functionals $\overline \F (\gamma_i \mid \mu_i)$ with $N$ soft functionals $\F(\gamma_i \mid \mu_i)$ defined in \eqref{F-func}.

Similarly, the coupled-two-marginal formulation of the Wasserstein barycenter problem can be written as 
\[
    {\rm W}_{\rm C2M}(\vec \mu) = \inf_{\nu \in \M(X)} \left\{\sum_{i=1}^N \lambda_i \inf_{\gamma_i \in \M(X \times X)}\left\{(c,\gamma_i) + \overline \F(\gamma_{i,1} \mid \mu_i) + \overline \F(\gamma_{i,2} \mid \nu)\right\}\right\}.
\]

By directly comparing with the corresponding unconstrained setting, given by ${\rm HK}_{\rm C2M}(\vec \mu)$ and defined in \eqref{eqn-def-HK-C2M}, we see that the difference lies in replacing $2N$ hard-constraint entropy functionals, $\overline \F(\gamma_{i,1}\mid \mu_i)$ and $\overline \F(\gamma_{i,2}\mid \nu)$, with $\F(\gamma_{i,1}\mid \mu_i)$ and $\F(\gamma_{i,2}\mid \nu)$.

On the other hand, in the constrained setting, given by ${\rm HK}_{\rm CC2M}(\vec \mu)$ and defined in \eqref{eqn-def-HK-CC2M}, $N$ hard-constraint entropy functionals $\overline \F(\gamma_{i,1}\mid \mu_i)$ get replaced with ${\F(\gamma_{i,1}\mid \mu_i)}$. In that sense the constrained setting can be seen as being half-way between the Wasserstein barycenter problem and the unconstrained HK barycenter problem. To a significant degree it is the hard constraint that $\gamma_{i,2} = \nu$, that ensures the link between the multi-marginal and coupled-two-marginal formulations for both the Wasserstein and constrained HK {\cb barycenter} can be established. Our result indicate that we thus retain more of the Wasserstein barycenter framework by using the constrained approach.

\section{Conclusion}
We have introduced a new conic multi-marginal barycenter problem related to the Hellinger-Kantorovich metric, which admits a very natural soft multi-marginal formulation which in its spirit exactly matches its Wasserstein metric counterpart. Additionally, we have shown that it corresponds to a coupled-two-marginal formulation with a one-sided hard marginal constraint, which can be argued to be a more natural extension of the idea of a Wasserstein barycenter to the unbalanced setting. Furthermore, our analysis of the conic multi-marginal formulation has established an explicit link between the unbalanced multi-marignal optimal transport framework studied in \cite{BvLNS23} and conic formulation of the HK barycenter problem studied in \cite{FMS21,CP21,BMS23}. We further have extended the results of \cite{BvLNS23} by showing that, as in the Wasserstein metric, the soft multi-marginal formulation can be posed over the space $X^N$ with the cost function defined as "the least cost" \cite{CE10}. 

In the future it would be very interesting to explore whether, based on the ideas developed here (see especially Section~\ref{sec:HKvsW-bary}), the unconstrained HK barycenter problem admits a "generalised" soft-marginal formulation.

\section*{Acknowledgments}
This work was inspired by a great talk on this topic given by Bernhard Schmitzer at the ICMS Workshop: Optimal Transport and Calculus of Variations in December 2023 in Edinburgh, whom I would like to sincerely thank for bringing my attention to this problem and for the discussion we have had on related topics. 

\section{Proofs}\label{sec:proofs}
\subsection{Proof of Lemma~\ref{lem-1}}\label{sec-proof-lemma}
\begin{proof}
Given the explicit form of $R$ in \eqref{eqn-def-R}, for a fixed {\cb $x$, $\vec x$, $s$, and $\vec s$}, we can solve 
\[
    \inf_{t > 0}\sum_{i=1}^N \Big\{t\big(\lambda_i c(x_i,x) + \lambda_i R\left(\tfrac{s_i}{t}\right) + \lambda_i R\left(\tfrac{s}{t}\right)\big)\Big\} 
\]
via a direct computation relying on differentiating in $t$. It readily follows that
\[
    \overline{H}_{\rm CMM}(\vec x, \vec s) = \inf_{(x,s) \in Y} \left( \sum_{i=1}^N \lambda_i s_i + s - 2s^{\tfrac12}\left(\prod_{{\cb j}=1}^N s_{\cb j}^{\tfrac{\lambda_{\cb j}}{2}}\right)\exp\left(\frac{-\sum_{{\cb k}=1}^N \lambda_{\cb k} c(x_{\cb k},x)}{2}\right)\right).
\]
The next step is to recognise that since $s_i \geq 0$ and $s \geq 0$, the monotonicity of the exponential function ensures that we also have 
\begin{equation}\label{eqn-HMM-inf-s}
    \overline{H}_{\rm CMM}(\vec x, \vec s)\hspace*{-2pt} = \hspace*{-2pt}\inf_{s \in \R_+} \left( \sum_{i=1}^N \lambda_i s_i + s - 2s^{\tfrac12}\left(\prod_{{\cb j}=1}^N s_{\cb j}^{\tfrac{\lambda_{\cb j}}{2}}\right)\exp\left(\cb \hspace*{-2pt} -\frac12\sum_{k=1}^N \lambda_k \inf_{x \in X}c(x_k,x)\right)\right).
\end{equation}
The infimum over $s$ can then again be solved directly by differentiating in $s$ and recognising that the optimal $s$ is given by
\[
    s = \left( \left(\prod_{{\cb j}=1}^N s_{\cb j}^{\tfrac{\lambda_{\cb j}}{2}}\right)\exp\left(\frac{-\inf_{x \in X} \sum_{{\cb k}=1}^N \lambda_{\cb k} c(x_{\cb k},x)}{2}\right) \right)^2.
\]
By plugging this formula into \eqref{eqn-HMM-inf-s} and simplifying, we arrive at the final reformulation, given by 
\[
    \overline{H}_{\rm CMM}(\vec x, \vec s) %= \sum_{i=1}^N \lambda_i s_i - \prod_{i=1}^N s_i^{\lambda_i}\exp\left(-\inf_{x \in X} \sum_{i=1}^N \lambda_i c(x_i,x)\right).
    = \sum_{i=1}^N \lambda_i s_i - \prod_{{\cb j}=1}^N s_{\cb j}^{\lambda_{\cb j}}\exp\left(-\inf_{x \in X} \sum_{{\cb k}=1}^N \lambda_{\cb k} c(x_{\cb k},x)\right).
\]
\end{proof}
\subsection{Proof of Proposition~\ref{prop-SMM-sol}}\label{sec-proof-SMM-sol}
\begin{proof}
    Define $\E(\cdot \mid \vec \mu)\,\colon \M(X^N) \to \overline{R}$ as
    \[
        \E(\gamma \mid \vec \mu):= (\tilde c, \gamma) + \sum_{i=1}^N \lambda_i\F(\gamma_i \mid \mu_i).
    \]
    The functional is feasible since when $\gamma = 0$ (the zero measure), we have 
    \[
    \E(0 \mid \vec \mu) = \sum_{i=1}^N \mu_i(X) < \infty.
    \] 
    It is bounded below by zero, thus a minimising sequence $(\gamma_k)_{k \in \mathbb N}$ exists. From the inequality 
        \[
        \E(\gamma \mid \vec \mu) \geq \gamma(X^N) \inf \tilde c + \sum_{i=1}^N \mu_i(X) F\left( \frac{\gamma(X^N)}{\mu_i(X)}\right) \geq \sum_{i=1}^N \mu_i(X) F\left( \frac{\gamma(X^N)}{\mu_i(X)}\right)
    \]
    and the superlinearity of $F$, we can deduce that the masses of $(\gamma_k)_{k \in \mathbb N}$ are bounded, thus the sequence has a weak-* cluster point $\bar \gamma$, meaning that, up to a subsequence, for all continuous and bounded $f$,
    \[
        \int_{X^N} f(\vec x) d\gamma_k(\vec x) \to \int_{X^N} f(\vec x) d\bar \gamma(\vec x).
    \]
    But it has to be a minimiser, since $\tilde c$ and $\F$ are both lower semi-continous. 
\end{proof}
\subsection{Proof of Theorem~\ref{thm-1}}\label{sec-proof-thm-1}
\begin{proof}
The proof of this theorem closely resembles the results detailed in \cite[Section 3]{BD23}, which themselves are inspired by the results presented in \cite[Section~3]{LMS18}. Recall from \eqref{eqn-def-HK-SMM} that the proposed soft multi-marginal (SMM) formulation is given by 
\[
    {\rm HK}_{\rm SMM}(\vec \mu) = \inf_{\gamma \in \M(X^N)} \left\{(\tilde c, \gamma) + \sum_{i} \lambda_i \F(\gamma_i \mid \mu_i)\right\}.
\]
Our aim is to establish that ${\rm HK}_{\rm SMM}(\vec \mu) = \overline{\rm HK}_{\rm CMM}(\vec \mu)$. Using the reverse entropies $R$ from \eqref{eqn-def-R} and the Lebesgue decomposition between $\gamma_i$ and $\mu_i$, 
\[
    \gamma_i = \sigma_i \mu_i + \gamma_i^{\perp}, \quad \mu_i = \rho_i \gamma_i + \mu_i^{\perp}
\]
it follows from an obvious adjustment of \cite[Theorem~3.11]{LMS18} that the problem admits a reverse formulation, given by
\[
{\rm HK}_{\rm SMM}(\vec \mu) = \inf_{\gamma \in \M(X^N)}\underbrace{\left\{\int_{X^N}\left(\tilde c(\vec x) + \sum_{i=1}^N \lambda_i R(\rho_i(x_i))\right) d\gamma + \sum_{i=1}^N \lambda_i (\mu_i - \rho_i\gamma_i)(X)\right\}}_{=: \Rc(\vec \mu, \gamma)},  
\]
where we have used the fact that $R_{\infty}' = F(0) = 1$. 

Likewise, we can replicate the argument in \cite[Proposition~4.3, Theorem~4.11]{LMS18} to conclude that the dual formulation is given by 
\[
    {\rm HK}_{\rm SMM}(\vec \mu) = \sup_{\vec \psi \in \Psi}\left\{ \sum_{i=1}^N (\psi_i,\mu_i)\right\} = \sup_{\vec \phi \in \Phi}\left\{ \sum_{i=1}^N (-(\lambda_i F)^*(-\phi_i),\mu_i)\right\},
\]
where
\[
    \vec \psi = (\psi_1,\dots, \psi_N), \quad \vec \phi = (\phi_1,\dots,\phi_N)
\]
belong to 
\begin{align*}
    \Psi &= \left\{ \vec \psi \,\mid\, \psi_i \in C_b(X), \; (\lambda_i R)^*(\psi_i) \in C_b(X),\; \bigoplus_{i=1}^N (\lambda_i R)^*(\psi_i) \leq \tilde c\right\},\\
    \Phi &= \left\{ \vec \phi \,\mid\, \phi_i \in C_b(X), \; -(\lambda _i F)^*(-\phi_i) \in C_b(X),\; \bigoplus_{i=1}^N \phi_i \leq \tilde c\right\}.
\end{align*}
Note that $(\lambda_i F)^*$ is the Legendre dual of $\lambda_i F$, which is defined, for a generic scalar function $f$, as 
\[
f^*(\phi) := \sup_{s >0}\left(s\phi - f(s)\right) \implies (\lambda_i f)^*(\phi) = \lambda_i f^*\left(\frac{\phi}{\lambda_i}\right).
\]
In our specific case we have
\[
    F^*(\phi) = \exp(\phi) - 1, \quad R^*(\psi) = -\log(1-\psi)
\]
and we note the change of variables relation
\[
    \phi_i = (\lambda_i R)^*(\psi_i), \quad \psi_i = -(\lambda_ i F)^*(-\phi_i).
\]

Finally, we introduce the marginal perspective function 
\[
    H_{\rm SMM}(\vec x, \vec s) := \inf_{t >0} \left\{ t\left(\sum_{i=1}^N \lambda_i R\left(\frac{s_i}{t}\right) + \tilde c(\vec x)\right)\right\},
\]
which, by a direct calculation, admits a formula
\[
    H_{\rm SMM}(\vec x, \vec s) = \sum_{i=1}^N \lambda_i s_i - \prod_{{\cb j}=1}^N s_{\cb j}^{\lambda_{\cb j}}\exp\left(-\tilde c(\vec x)\right).
\]
In other words, since $\tilde c(\vec x) = \inf_{x \in X} \sum_{i=1}^N \lambda_i c(x_i,x)$, we have established that in fact $H_{\rm SMM} = \overline H_{\rm CMM}$ from \eqref{eqn-def-H-MM}.

We also have a dual representation for $H_{\rm SMM}$, namely
\begin{align*}
    H_{\rm SMM}(\vec x, \vec s) &= \sup_{\vec \psi} \left\{ \sum_{i=1}^N s_i \psi_i \, \Bigm\vert \, \sum_{i=1}^N(\lambda_i R)^*(\psi_i) \leq \tilde c(\vec x) \right\}\\
    &= \sup_{\vec \phi} \left\{ \sum_{i-1}^N -s_i (\lambda_i F)^*(-\phi_i)\, \Bigm\vert \, \sum_{i=1}^N \phi_i \leq \tilde c(\vec x) \right\},
\end{align*}
where now $\phi_i$ and $\psi_i$ are scalars. 

We can thus introduce a homogeneous formulation, given by
\begin{equation}\label{eqn-def-hom-form-XN}
    \H(\vec \mu \mid \gamma) := \int_{X^N}H_{\rm SMM}(\vec x, \rho_1(x_1),\dots,\rho_N(x_N)) d\gamma + \sum_{i=1}^N \lambda_i (\mu_i - \rho_i \gamma_i)(X)
\end{equation}
and, using the strategy from \cite[Theorem~5.5]{LMS18}, it follows from the dual representation for $H_{\rm SMM}$ that  
\begin{equation}\label{eqn-sandwich}
    \underbrace{\sup_{\vec \phi \in \Phi}\left\{ \sum_{i=1}^N (-(\lambda_i F)^*(-\phi_i),\mu_i)\right\}}_{={\rm HK}_{\rm SMM}(\vec \mu)}\leq \inf_{\gamma \in \M(X^N)} \H(\vec \mu \mid \gamma) \leq \underbrace{\inf_{\gamma \in \M(X^N)}\Rc(\vec \mu \mid \gamma)}_{={\rm HK}_{\rm SMM}(\vec \mu)},
\end{equation}
which of course imply that the middle term is equal to ${\rm HK}_{\rm SMM}(\vec \mu)$ too. 

The final step is to lift to the cone. This is done by defining 
\[
    \inf_{\alpha \in S(\vec \mu)} (H_{\rm SMM}, \alpha),
\]
where we recall from \eqref{eqn-def-S-vec-mu} and \eqref{eqn-def-h-marg} that
\[
    S(\vec \mu) := \{ \alpha \in \M(Y^N)\,\mid\, \h_i \alpha = \mu_i\}, \quad \h_i\alpha = \pi^{x_i}_{\#}(s_i \alpha).
\]
This is the cone-equivalent of \eqref{eqn-def-hom-form-XN}, especially since, as in \eqref{R-inf-zero}, the singular part is zero due to the constraint on the homogeneous marginal in the set $S(\vec \mu)$. 

Obvious adjustments to the argument in \cite[Theorem 5.8]{LMS18} allow us to conclude that in fact the sandwich inequality in \eqref{eqn-sandwich} holds for the cone formulation too. Since we have already shown that $H_{\rm SMM} = \overline{H}_{\rm CMM}$, we thus obtain that 
\[
    {\rm HK}_{\rm SMM}(\vec \mu) = \inf_{\alpha \in S(\vec \mu)} (\overline{H}_{\rm CMM}, \alpha) = \overline{\rm HK}_{\rm CMM}(\vec \mu),
\]
which is what we set out to prove. 

The fact that
\[
    \bar \alpha = (x_1,\bar \rho_1(x_1),\dots,x_N, \rho_N(x_N))_{\#}\bar \gamma, 
\]
where 
\[
    \mu_i = \bar \rho_i \bar \gamma_i + \mu_i^{\perp},
\]
is a minimiser of $\overline{\rm HK}_{\rm CMM}(\vec \mu)$ is an immediate extension of \cite[Theorem~5.8]{LMS18}.
\end{proof}
\subsection{Proof of Proposition~\ref{prop-CC2M-sol}}\label{sec:proof-CC2M-sol}
\begin{proof}
The argument is similar to the proof {\cb in} \cite[Proposition~4.2]{FMS21}. Define
\[
    J\,\colon \M(X) \to \overline \R
\]
given by 
\[
    J(\nu) := \sum_{i=1}^N \lambda_i \overline{\rm HK}^2(\mu_i,\nu).
\]
The functional $J$ is feasible, that is $\inf_{\nu \in \M(X)} J(\nu) < \infty$. This follows naturally from the fact that the zero measure is a competitor and 
\[
    J(0) = \sum_{i=1}^N \lambda_i \overline{\rm HK}^2(\mu_i,0) = \sum_{i=1}^N \lambda_i \mu_i(X) < \infty.
\]
The functional $J$ is also bounded below. This follows from 
\[
    J(v) = \sum_{i=1}^N \lambda_i \overline{\rm HK}^2(\mu_i,\nu) \geq \sum_{i=1}^N \lambda_i {\rm HK}^2(\mu_i,\nu) \geq 0,
\]
where the last inequality holds since ${\rm HK}$ is a distance. The first inequality follows from the fact that the minimizer of $\overline{\rm HK}(\nu_i,\mu)$ is a competitor in ${\rm HK}(\nu_i,\mu)$ and further that $F(1) = 0$.

Thus a minimising sequence $(\nu_k)_{k \in \mathbb N}$ exists and from the lower bound
\[
    \overline{\rm HK}^2(\mu, \nu) \geq {\rm HK}^2(\mu,\nu) \geq (\sqrt{\mu(X)} - \sqrt{\nu(X)})^2,
\]
we can conclude that $\nu_k(X)$ is uniformly bounded, and hence the sequence has a weak-* cluster point $\bar \nu$, meaning that, up to a subsequence, for all continuous and bounded $f$,
\[
    \int_{X} f(x) d\nu_k(x) \to \int_X f(x) d\bar \nu(x).
\]

At the same time, due to \cite[Theorem~3.3]{LMS18}, for each $\nu_k$, there exists a collection $\{\gamma^{i,k}\}_{i=1}^N \subset \M(X \times X)$ of minimizers of $\overline{\rm HK}^2(\mu_i,\nu_k)$. Using the notation 
\[
    \E(\gamma \mid \nu, \mu) := (c,\gamma) + \F(\gamma_1 \mid \mu) + \overline{\F}(\gamma_2 \mid \nu),
\]
we thus have 
\[
    J(\nu_k) = \sum_{i=1}^N \lambda_i \E(\gamma^{i,k} \mid \mu_i, \nu_k).
\]
In particular, the hard-constraint entropy functional $\overline \F(\gamma^{i,k}_2 \mid \nu_k)$ ensures that
\[
    \gamma^{i,k}(X \times X) = \gamma^{i,k}_2(X) = \nu_k(X),
\] 
implying that the masses of $\{\gamma^{i,k}\}_{k \in \mathbb N}$ are uniformly bounded, thus ensuring that the sequence admits a weak-* cluster point $\bar \gamma^i$, meaning that, up to a subsequence, for all continuous and bounded $f$,
\[
    \int_{X \times X} f(x_1,x_2)d\gamma^{i,k}(x_1,x_2) \to \int_{X \times X}f(x_1,x_2)d\bar \gamma^i(x_1,x_2).
\]
The final step is to realise that, up to choosing a subsequence, we have
\begin{align*}
    \inf_{\nu \in \M(X)}J(\nu) &= \lim_{k\to\infty}J(\nu_k) = \lim_{k\to \infty} \sum_{i=1}^N \lambda_i \E(\gamma^{i,k}\mid \mu_i, \nu_k)\\ &\geq \sum_{i=1}^N \lambda_i \E(\bar \gamma^i\mid \mu_i, \bar \nu) \geq \sum_{i=1}^N \lambda_i \overline{\rm HK}^2(\mu_i,\bar\nu) \geq \inf_{\nu \in \M(X)}J(\nu). 
\end{align*}
The first inequality follows from the fact that the cost function $(c,\gamma^{i,k})$ and the entropy functional $\F(\gamma^{i,k}_1 \mid \mu_i)$ are both lower semi-continuous, whereas $\overline\F(\gamma^{i,k}_2 \mid \nu_k)$ is (trivially) jointly lower semi-continuous in both variables, (c.f. \cite[Corollary 2.9]{LMS18}). The second inequality follows merely from the fact that $\bar \gamma^i$ is a competitor. We have thus established that $\bar \nu$ is a minimiser. 
\end{proof}
\subsection{Proof of Theorem~\ref{thm-2}}\label{sec-proof-thm-2}
We begin by proving an auxiliary result concerning the characterisation of the cost function
\[
    \tilde c(\vec x) = \inf_{x \in X} \sum_{i=1}^N \lambda_i c(x_i,x),
\]
defined in \eqref{eqn-def-c-tilde}.
\begin{lemma}\label{lem-c-tilde}
    The function $\tilde c$ attains a finite value only on the set
    \[
       A := \left\{ \vec x \in X^N \mid \bigcap_{i=1}^N B_{\frac{\pi}{2}(x_i)} \neq \emptyset\right\}  
    \]
    and is identically $+\infty$ otherwise. Furthermore, if $\vec x \in A$, then there exists a unique minimizer of 
    \[
        \inf_{x \in X} \sum_{i=1}^N \lambda_i c(x_i,x)
    \]
    and consequently there exists a map $T\,\colon A \to X$ mapping $\vec x$ to the minimizer.
\end{lemma}
\begin{proof}
    The fact that $\tilde c(\vec x) < \infty$ if and only if $\vec x \in A$ follows directly from the definition of $c$ and $\tilde c$. We note that $A$ is convex. Now let us fix $\vec x \in A$ and define $f\,\colon\,B \to \R$, where 
    \[
        B:= \bigcap_{i=1}^N B_{\frac{\pi}{2}}(x_i), \quad f(x) := \sum_{i=1}^N \lambda_i c(x_i,x).
    \] 
    By design, $f$ is finite on its domain and $B$ is a non-empty convex set, since it is a finite intersection of balls. We can thus rewrite $f$ as
\begin{align*}
    f(x) &= -\sum_{i=1}^N \lambda_i \log(\cos^2(|x_i-x|)) = -\log\left(\prod_{i=1}^N \cos^{2\lambda_i}(|x_i-x|)\right)
\end{align*}
It can be readily shown that 
\[
    g\,\colon B_{\frac{\pi}{2}}(0) \to \R, \quad g(x) = -\log(\cos^2(|x|))
\]
is strictly convex, and thus so is $f$, since it is a sum of convex functions and $B$ is convex. As a result, it has a unique minimiser, thus establishing existence of the mapping $T$.  
\end{proof}

We can now proceed to proving the main theorem of this section. 
\begin{proof}[Proof of Theorem~\ref{thm-2}]
    As in the corresponding proof in the Wasserstein metric \cite{AC11}, we will proceed in two steps, but first, in light of Lemma~\ref{lem-c-tilde}, we discuss the special case where %out of the collection of measures $\vec \mu = (\mu_1, \dots, \mu_N)$, there exists at least one pair of indices $i, j \in \{1,\dots,N\}$ such that \cmb{this is wrong:}
    \[
        \bigcap_{i=1}^N B_{\frac{\pi}{2}}({\rm supp}\mu_i) = \emptyset,
     \]
    where we use the notation 
    \[
        B_{\frac{\pi}{2}}(A):= \left\{x \in X \mid \d(x,A) < \frac{\pi}{2}\right\}, \quad \d(x,A):= \inf_{y \in A} |x-y|.
    \]
    This is the case where no transport of mass can occur. This is, on the one hand, because the minimizer of ${\rm HK}_{\rm CC2M}(\vec \mu)$ is the zero measure, that is $\bar \nu = 0$, which is the consequence of the hard marginal constraint $\bar \gamma_{i,2} = \bar \nu$. As a result only the zero measure ensures that $\overline{\rm HK}^2(\mu_i,\nu)$ is finite for all $i$. Consequently, each two-marginal problem is minimized by the zero measure too, that is $\bar \gamma_i = 0$. 
    
    On the other hand, in the soft marginal formulation, the minimiser of ${\rm HK}_{\rm SMM}(\vec \mu)$ is given by $\bar \gamma = 0$, which follows from the fact that any feasible plan $\gamma \in \M(X^N)$ has to satisfy $\gamma_i \ll \mu_i$, from which it follows that $(\tilde c, \gamma) = +\infty$ unless $\gamma = 0$. In this case the mapping $T$ does not exist, but it does not need to be discussed since the pushforward of the zero measure is always the zero measure.

    In what follows we thus focus on the other case, namely that
    \begin{equation}\label{eqn-supp-mu_i-constraint}
        \bigcap_{i=1}^N B_{\frac{\pi}{2}}({\rm supp}\mu_i) \neq \emptyset.
    \end{equation}

\textbf{Step 1: showing that ${\rm HK}_{\rm CC2M}(\vec \mu) \geq {\rm HK}_{\rm SMM}(\vec \mu)$.} Let $\bar \nu \in \M(X)$ be a minimizer of ${\rm HK}_{\rm CC2M}(\vec \mu)$ (proven to exist in Proposition~\ref{prop-CC2M-sol}) and $\bar \gamma_i \in \M(X \times X)$ be a minimizer of $\overline{\rm HK}^2(\mu_i,\bar \nu)$. 
By $(\bar \gamma_i^{x_2})_{x_2 \in X}$ we denote the disintegration of $\bar \gamma _i$ with respect to its second marginal, which, by construction, for each $i$ is given by $\bar \gamma_{i,2} = \bar \nu$.

This allows us to introduce $\hat \gamma \in \M(X^N)$ via
\[
    \int_{X^N}\phi(\vec x) d \hat \gamma := \int_{X^{N+1}} \phi(\vec x)d\bar\gamma^y_1(x_1),\dots,d\bar\gamma^y_{N}(x_N)d\bar \nu(y).
\]
We note that the $i$th marginal of $\hat \gamma_i$ is given by $\hat \gamma_i = \bar \gamma_{i,1}$, which follows from observing that 
\begin{align*}
    \int_X \phi(x)d\hat \gamma_i &= \int_{X^N}\phi(\pi^i(\vec x))d\hat \gamma = \int_{X \times X} \phi(x_i)d\hat \gamma_i^{y}(x_i)d\bar \nu = \int_{X\times X} \phi(x_1) d\bar \gamma_i\\ 
    &= \int_X \phi(x) d\bar \gamma_{i,1}.
\end{align*}
This in particular naturally implies that in the Lebesgue decomposition between $\hat \gamma_i$ and $\mu_i$, 
\[
    \mu_i = \hat\rho_{i} \hat \gamma_i + \mu_i^{\perp},
\]
we have $\hat \rho_i = \bar \rho_{i,1}$. As a result
\begin{align*}
    &{\rm HK}_{\rm SMM}(\vec \mu) \leq (\tilde c,\hat \gamma) + \sum_{i=1}^N \lambda_i \F(\hat \gamma_i \mid \mu_i)\\
    &= \int_{X^N}\left(\tilde c(\vec x) + \sum_{i=1}^N \lambda_i R(\hat \rho_i(x_i))\right) d\hat \gamma + \sum_{i=1}^N \lambda_i (\mu_i -  \hat\rho_i \hat\gamma_i)(X)\\
    &= \int_{X^{N+1}}\left( \inf_{x \in X} \sum_{i=1}^N \lambda_i c(x_i,x) + \sum_{i=1}^N \lambda_i R(\bar \rho_{i,1}(x_i)) \right)d\hat \gamma  + \sum_{i=1}^N \lambda_i (\mu_i -  \bar\rho_{i,1} \bar\gamma_{i,1})(X)\\
    &\leq \int_{X^{N+1}}\left( \sum_{i=1}^N \lambda_i c(x_i,y) + \sum_{i=1}^N \lambda_i R(\bar \rho_{i,1}(x_i)) \right)d\hat \gamma  + \sum_{i=1}^N \lambda_i (\mu_i -  \bar\rho_{i,1} \bar\gamma_{i,1})(X)\\
     &{\cb = \sum_{i=1}^N\left( \lambda_i \int_{X\times X}\left(c(x_i,y) + R(\bar \rho_{i,1}(x_i))\right)d\bar\gamma_i^y(x_i)d\bar \nu(y)\right) + \sum_{i=1}^N \lambda_i (\mu_i -  \bar\rho_{i,1} \bar\gamma_{i,1})(X)} \\
    &{\cb = \sum_{i=1}^N \left(\lambda_i \int_{X\times X}c(x_i,y) + R(\bar \rho_{i,1}(x_i))d\bar\gamma_i + \lambda_i (\mu_i -  \bar\rho_{i,1} \bar\gamma_{i,1})(X)\right)}\\
    &= \sum_{i=1}^N \lambda_i (c,\bar \gamma_i) + \lambda_i \F(\bar \gamma_i \mid \mu_i)\\
    &= \sum_{i=1}^N \lambda_i \overline{\rm HK}^2(\mu_i,\bar\nu) = {\rm HK}_{\rm CC2M}(\vec \mu).
\end{align*}

\textbf{Step 2: showing that ${\rm HK}_{\rm CC2M}(\vec \mu) \leq {\rm HK}_{\rm SMM}(\vec \mu)$.} Let $\bar \gamma \in \M(X^N)$ be the minimiser of ${\rm HK}_{\rm SMM}(\vec \mu)$ and $T\,\colon\,A \to X$ be the mapping to the infimum in the cost function $\tilde c$, which exists due to the assumption on the support of measures $\mu_1,\dots, \mu_N$ in \eqref{eqn-supp-mu_i-constraint}, as detailed in Lemma~\ref{lem-c-tilde}. We thus have 
\[
    \tilde c (\vec x) = \inf_{x \in X}\sum_{i=1}^N \lambda_i c(x_i,x) = \sum_{i=1}^N \lambda_i c(x_i,T(\vec x)).
\]
We can then set $\hat \gamma_i := (\pi^i, T)_{\#}\bar \gamma$ and $\hat \nu := T_{\#}\bar \gamma$. It readily follows that 
\[
    \hat \gamma_{i,1} = \bar \gamma_i, \quad \hat \gamma_{i,2} = \hat \nu,
\]
which allows us to argue that
\begin{align*}
    {\rm HK}_{\rm SMM}(\vec \mu) &= \sum_{i=1}^N \lambda_i \left( \int_{X^N} c(x_i, T(\vec x)) d \bar \gamma + \F(\bar \gamma_i \mid \mu_i) \right)\\
    &= \sum_{i=1}^N \lambda_i \left( (c, \hat \gamma_i) + \F(\hat \gamma_{i,1} \mid \mu_i ) + \overline \F(\hat \gamma_{i,2} \mid \hat \nu)\right)\\
    &\geq {\rm HK}_{\rm CC2M}(\vec \mu),
\end{align*}
where the second equality follows from the fact that by construction $\hat \gamma_{i,2} = \hat \nu$ and thus $\overline F(\hat \gamma_{i,2} \mid \hat \nu) = 0$. The final inequality follows from the fact that $\hat \gamma_{i}$ and $\hat \nu$ are admissible competitors.
\end{proof}

\bibliographystyle{siamplain}
\bibliography{papers}

\end{document}